\documentclass[12pt]{article}
\usepackage{amsthm}
\usepackage{amsmath}
\usepackage{hyperref}
\usepackage{tikz-cd}
\usepackage{subcaption}
\usepackage{graphicx}
\usepackage{algorithm}
\usepackage{setspace}
\usepackage{thmtools,thm-restate}
\DeclareMathOperator{\lcm}{lcm}

\usepackage{array}   
\newcolumntype{C}{>{$}c<{$}} 

\newtheorem{theorem}{Theorem}
\theoremstyle{definition}
\newtheorem{defn}{Definition}

\theoremstyle{conjecture}
\newtheorem{conj}{Conjecture}
\theoremstyle{remark}

\theoremstyle{lemma}

\theoremstyle{proposition}

\theoremstyle{corollary}

\newcommand{\rr}{\textbf{r}}
\newcommand{\dd}{\textbf{d}}
\newcommand{\diag}{\textup{diag}}

\newcommand{\cyc}{\text{cyc}}
\newcommand{\ceil}[2]{\left \lceil \frac{#1}{#2} \right \rceil}
\newcommand{\floor}[2]{\left \lfloor \frac{#1}{#2} \right \rfloor}
\newcommand{\inter}[2]{[#1,\;#2)}
\newcommand{\row}{\textup{row}}
\newcommand{\column}{\textup{column}}
\renewcommand{\mod}[1]{\;\textup{mod} \;#1}
\allowdisplaybreaks

\title{
When Two-Holed Torus Graphs are Hamiltonian
}

\author{Dhruv Rohatgi}

\date{September 5, 2016}

\begin{document}

\maketitle

\begin{abstract}
Trotter and Erd{\"o}s found conditions for when a directed $m \times n$ grid graph on a torus is Hamiltonian. We consider the analogous graphs on a two-holed torus, and study their Hamiltonicity. We find an $\mathcal{O}(n^4)$ algorithm to determine the Hamiltonicity of one of these graphs and an $\mathcal{O}(\log(n))$ algorithm to find the number of diagonals, which are sets of vertices that force the directions of edges in any Hamiltonian cycle. We also show that there is a periodicity pattern in the graphs' Hamiltonicities if one of the sides of the grid is fixed; and we completely classify which graphs are Hamiltonian in the cases where $n=m$, $n=2$, the $m \times n$ graph has $1$ diagonal, or the $\frac{m}{2} \times \frac{n}{2}$ graph has $1$ diagonal.
\end{abstract}

\section{Introduction}

Most problems in graph theory arise from considering some graph property---planarity, connectedness, being Eulerian, or Hamiltonicity---and attempting to classify which graphs have that property. Almost all of these properties are motivated by real-world applications. To give a few examples, planar graphs are of interest in designing electrical circuits, and a Hamiltonian cycle---a closed path through the graph which visits every vertex exactly once---in computational biology represents the reconstruction of a DNA strand from its constituent $k$-mers. We will be focusing on that last graph property, Hamiltonicity. Though it has widespread scientific uses, its theory is not well understood. There is no simple classification of Hamiltonian graphs, and in fact the problem of finding a Hamiltonian cycle is NP-complete.


In certain cases, Hamiltonicity is equivalent to a simpler property. For instance, a directed graph where every vertex has exactly 1 in-edge and exactly 1 out-edge (i.e. a permutation graph) is Hamiltonian if and only if it is connected (i.e. the graph is composed of one cycle rather than multiple). In this paper we study a class of graphs where checking Hamiltonicity is equivalent to checking the connectedness of several permutation graphs (following Trotter and Erd{\"o}s' logic in \cite{erdos}), and discover that the number of these permutation graphs is polynomial. We consider rectangular grids of varying sizes which are folded into the shape of a two-holed torus, and draw directed edges up as well as right from each grid cell.

The case of a one-holed torus has been studied earlier, and solved by Trotter and Erd{\"o}s, who found a simple condition to classify all grid sizes yielding Hamiltonian graphs \cite{erdos}. Non-orientable surfaces have been studied as well. In particular, for grids folded into the shape of a projective plane, it is known which sizes produce a Hamiltonian graph after adding directed edges up and right from each cell (see \cite{projective1} and \cite{projective2}). However, grids folded into tori with multiple holes have not been classified, motivating our research problem.

We now outline the remainder of this paper.

In Section~\ref{sec:prelim}, we define the graphs with which we are working more precisely. Every two-holed torus graph is described by two positive integers $n$ and $m$.

In Section~\ref{sec:diagonals}, we define the \textit{diagonals} of a two-holed torus graph---sets of vertices which can be reached from each other by moving diagonally through the graph---and use their properties to prove several general results about the existence of Hamiltonian cycles. Specifically, in Section~\ref{subsec:diagonalbasics}, we define diagonals, construct an $\mathcal{O}(mn\cdot 16^{\gcd(m,n)})$ algorithm for determining Hamiltonicity, and show that the number-of-diagonals function commutes with scalar multiplication of the grid size vector. In Section~\ref{subsec:links}, we show that the problem of determining Hamiltonicity is related to the problem of determining whether a link on a two-holed torus is connected, and construct an $\mathcal{O}((m+n)\gcd(m,n)^3)$ algorithm for determining Hamiltonicity. Finally, in Section~\ref{subsec:periodicity}, we show that under certain conditions the property of Hamiltonicity is periodic.

In Section~\ref{sec:counting}, we develop an $\mathcal{O}(\log(n))$ algorithm to count the number of diagonals in a two-holed torus graph. Specifically, in Section~\ref{subsec:reductions} we find a set of $10$ equivalences which allow us to reduce a large size graph to a smaller size graph with the same number of diagonals, and in Section~\ref{subsec:ternary} we find that these reductions have a simpler form on a ternary tree enumerating all coprime pairs $(m,n)$.

In Section~\ref{sec:speccases}, we completely classify which two-holed torus graphs are Hamiltonian in several special cases---when $n=2$, when $n=m$, and when the graph has one diagonal or the graph  with parameters $\frac{n}{2}$ and $\frac{m}{2}$ has one diagonal.

\section{Preliminaries}\label{sec:prelim}

Let $G = (V,E)$ be a directed graph, where $V$ is the set of vertices and $E$ is the set of directed edges in the graph. Then $G$ is called \textit{Hamiltonian} if there is an ordering of the vertices, $v_1,v_2,\dots,v_{|V|}$, such that $(v_i,v_{i+1}) \in E$ for $1 \leq i < |V|$ and $(v_{|V|},v_1) \in E$---that is, there is an edge between consecutive vertices. See Figure~\ref{fig:hamiltonian} for examples.

\begin{figure}[h]
	\centering
	\begin{subfigure}[t]{0.35\textwidth}
		\centering
		$$\begin{tikzcd}[column sep = huge, row sep = huge]
		1 \arrow{r} & 2 \arrow{d} \arrow[color=gray]{dl}\\
		3 \arrow{u} & 4 \arrow{l}
		\end{tikzcd}$$
		\caption{Hamiltonian}
	\end{subfigure}
	\begin{subfigure}[t]{0.35\textwidth}
		\centering
		$$\begin{tikzcd}[column sep = huge, row sep = huge]
		1 \arrow{r} & 2 \arrow{dl}\\
		3 \arrow{u} & 4 \arrow{l}
		\end{tikzcd}$$
		\caption{Non-Hamiltonian}
	\end{subfigure}
	\caption{Examples of (non)-Hamiltonian graphs}\label{fig:hamiltonian}
\end{figure}
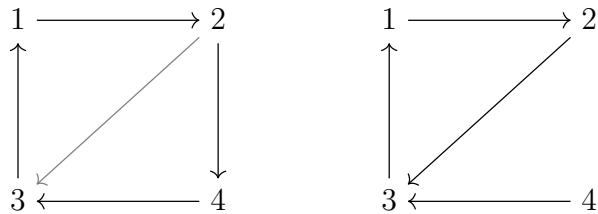

A useful operation for building more complex graphs from simpler graphs is the cartesian product, as defined by Trotter and Erd{\"o}s \cite{erdos}. One of the simplest classes of (directed) graphs is the set of directed cycles $\overrightarrow{C_n}$. Trotter and Erd{\"o}s considered the class of cartesian products of two directed cycles $\overrightarrow{C_m} \times \overrightarrow{C_n}$, and classified when these graphs are Hamiltonian. There are two concise ways to visualize a product of directed cycles. One of them is in three dimensional space: the graph is a ``ring of rings", the skeleton of a torus. The other is in the plane: the graph is an augmented, directed grid graph, with edges up and right and additional edges wrapping around the boundaries.

In subsequent research, Forbush et al. \cite{projective1} and McCarthy and Witte \cite{projective2} have generalized the problem from the three-dimensional, topological perspective. Note that the manner in which the edges of Trotter and Erd{\"o}s' grid graph wrap around corresponds to the identification mapping from a rectangle to a torus---that is, an assignment of direction to each boundary of the rectangle and a pairing of the boundaries so that when pairs are glued (identified) together so that their directions align, the resulting closed surface can be deformed into a torus. There is a similar identification map from the rectangle to the projective plane, and reassigning the boundary edges to correspond to this identification map produces an augmented grid graph with the ``shape" of a projective plane.

Tori with higher genus, however, have not been studied. There is no natural identification map from the rectangle to a $k$-holed torus where $k>1$. Rather, there is a map from the polygon with $4k$ sides to a $k$-holed torus (see Figure~\ref{fig:octagon}) \cite{topology}. Yet we wish to stick to rectangles. For surfaces more complex than the torus, an augmented grid graph in the shape of the surface may not be as simple as a cartesian product of graphs, but in constructing the $k$-holed torus graph, we wish to preserve some of the simplicity of the ``ring of rings" rectangular grid graph: there are two unambiguous directions, up and right.

\begin{figure}[h]
	\centering
	\begin{subfigure}[t]{0.35\textwidth}
		\centering
		\includegraphics[width=0.75\textwidth]{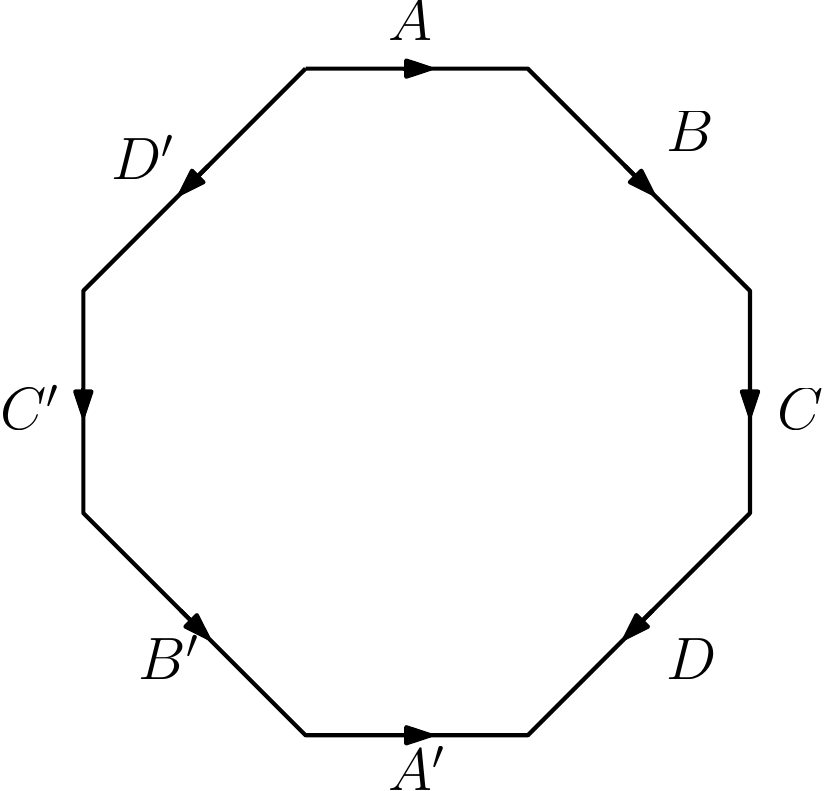}
		\caption{Unfolded}
	\end{subfigure}
	\begin{subfigure}[t]{0.35\textwidth}
		\centering
		\includegraphics[width=1.0\textwidth]{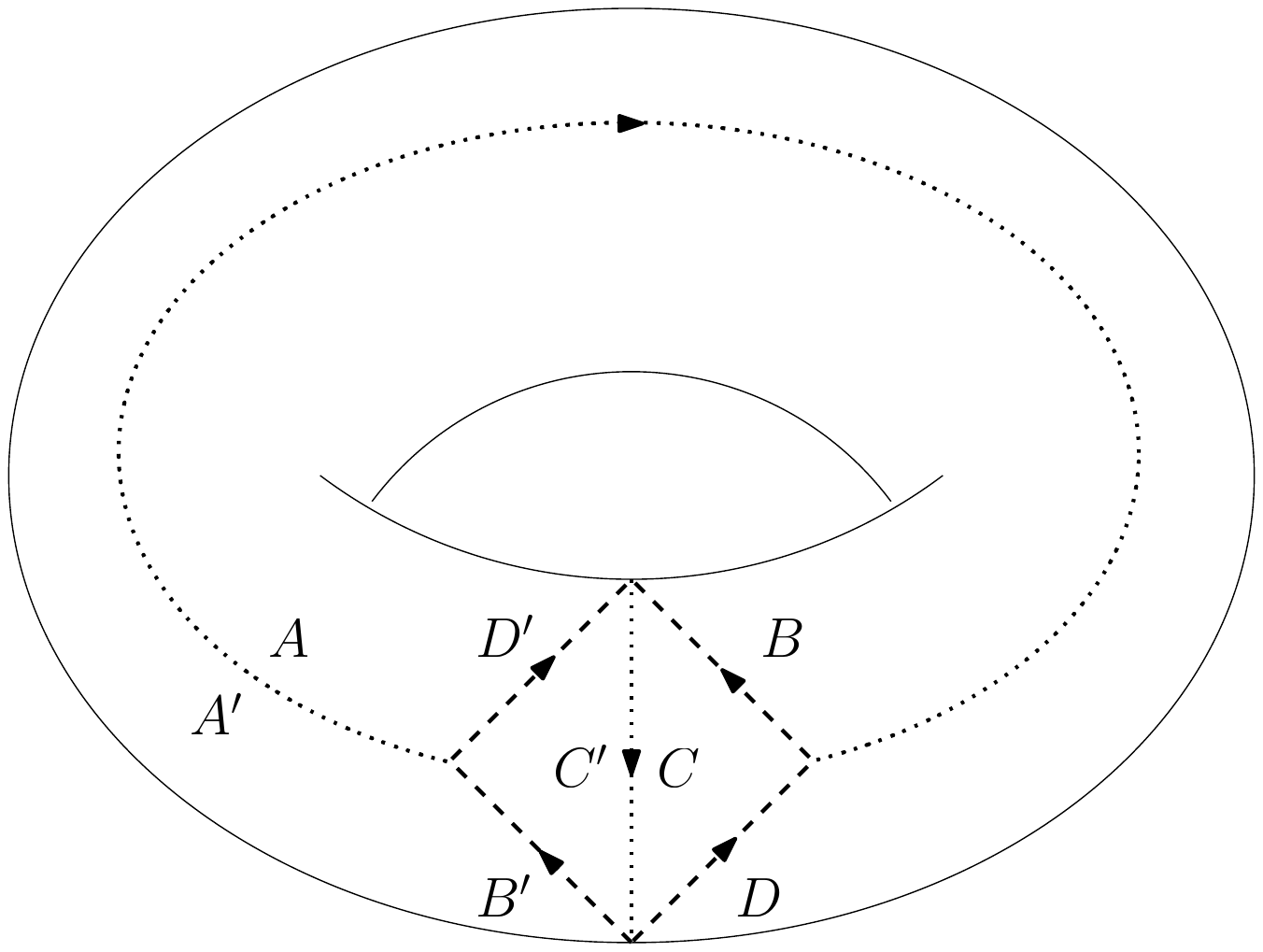}
		\caption{Partially Folded}
	\end{subfigure}
	\caption{The identification map of a two-holed torus}\label{fig:octagon}
\end{figure}



To fold a rectangle into the shape of a $k$-holed torus, we divide each boundary of the rectangle into $k$ segments, so the rectangle is effectively a polygon with $4k$ sides. The $i^\text{th}$ segment from the left of the upper boundary is identified with the $i^\text{th}$ segment from the right of the lower boundary (the diametrically opposite segment); and the $i^\text{th}$ segment from the top of the left boundary is identified with the $i^\text{th}$ segment from the bottom of the right boundary. See Figure~\ref{fig:idmap} for an example.

We can now place an augmented grid graph on this rectangle.


\begin{defn}
Let us consider the following construction of a graph $G_{n,m,k}$ in our class, where $n,m>1$ and $k\geq 1$: fold a $kn \times km$ rectangular grid into the shape of a $k$-holed torus, so that the $k$ segments on each boundary have equal length. Let the vertices of $G_{n,m,k}$ be the cells in the grid. There is a bijection $u$ from each cell to the cell above, and a bijection $r$ from each cell to the cell on the right. Let the directed edges of $G_{n,m,k}$ be $(\alpha,u\alpha)$ and $(\alpha,r\alpha)$ for each cell $\alpha$.
\end{defn}

\begin{figure}[h]
\centering
\includegraphics[width=0.3\textwidth]{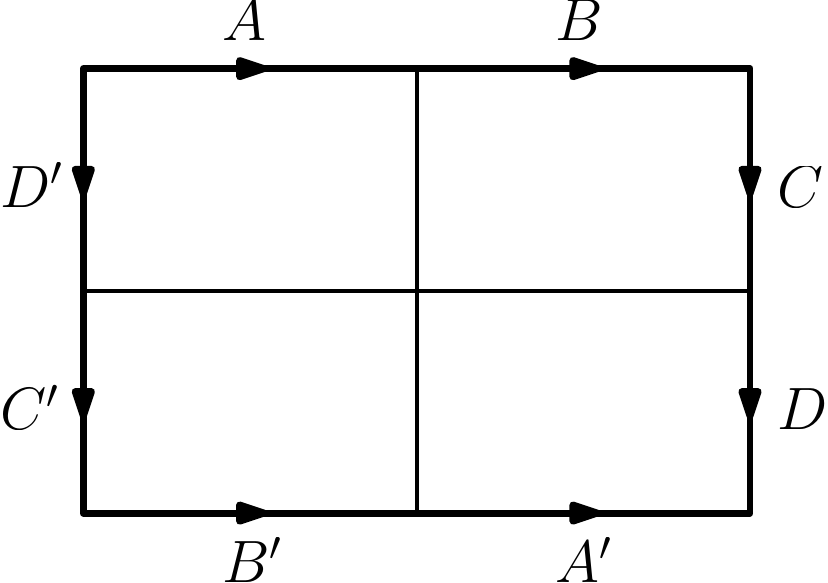}
\caption{The rectangular identification map of a two-holed torus}\label{fig:idmap}
\end{figure}

See Figure~\ref{fig:graph} for examples.

\begin{figure}[h]
\centering
\includegraphics[width=0.3\textwidth]{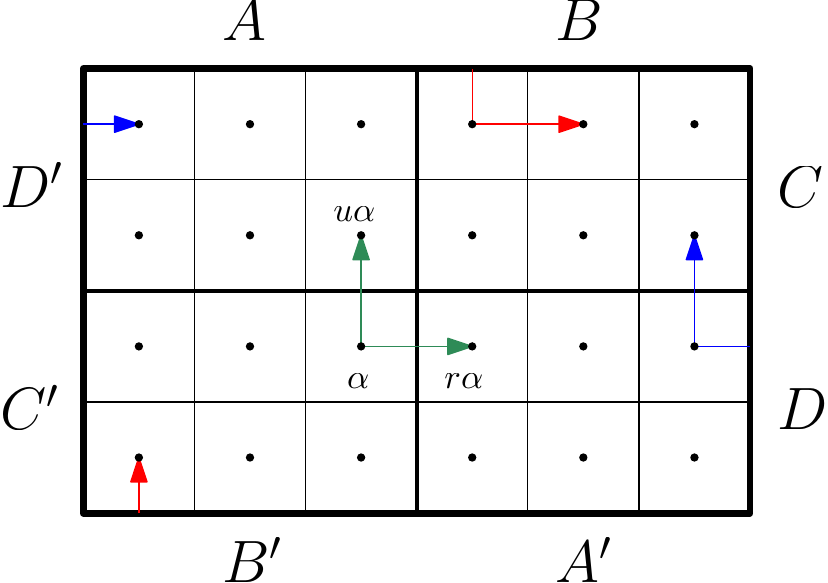}
\caption{Several outgoing edges of $G_{2,3,2}$ in different colors}\label{fig:graph}
\end{figure}

With this notation, Trotter and Erd{\"o}s' result can be rephrased as follows:

\begin{theorem}[Trotter and Erd{\"o}s]
Let $n,m \geq 1$. Let $g = \gcd(n,m)$. Then $G_{n,m,1}$ is Hamiltonian if and only if there are positive integers $g_1$ and $g_2$ such that $g_1+g_2=g$ and $\gcd(g_1,n) = \gcd(g_2,m) = 1$.
\end{theorem}

For the remainder of this paper, we will focus our attention on the simplest unsolved case, where $k=2$.

\section{Relating Hamiltonian Cycles to Grid Diagonals}\label{sec:diagonals}

\subsection{Grid Diagonals}\label{subsec:diagonalbasics}

We first show that a Hamiltonian cycle in any of our graphs must be in some sense regular; the grid can be partitioned into certain sets, which we call \textit{diagonals}, and in the Hamiltonian cycle the edges out of cells (i.e. vertices) in the same diagonal are in the same direction.

\begin{defn}
For any positive integers $n$ and $m$, define a \textit{diagonal} in $G_{n,m,2}$ to be a set $$\{\alpha, u^{-1}r\alpha, (u^{-1}r)^2\alpha, \dots \}$$ for any cell $\alpha$ (see Figure~\ref{fig:diagonal}).
\end{defn}

The same notion of a diagonal can be found in work on the projective checkerboard (see \cite{projective1} and \cite{projective2}). Note that different cells $\alpha_1$ and $\alpha_2$ may generate the same diagonal, if $\alpha_1 = (u^{-1}r)^i \alpha_2$ for some $i \geq 0$.

\begin{restatable}{lemma}{rightdown}\label{lemma:rightdown}
Let $n$ and $m$ be positive integers. Suppose $G_{n,m,2}$ is Hamiltonian, and pick any Hamiltonian cycle. In any one diagonal, the out-edges of the cells in this diagonal which are used in the Hamiltonian cycle point in the same direction.
\end{restatable}

\begin{proof}
Suppose the edges out of two cells $\alpha$ and $u^{-1}r\alpha$ are in different directions. If the edge from $\alpha$ points up and the edge from $u^{-1}r\alpha$ points right, then neither in-edge of $r\alpha$ is in the cycle --- this is a contradiction. If the edge from $\alpha$ points right and $u^{-1}r\alpha$ points up, then both in-edges of $r\alpha$ are in the cycle --- again, contradiction. Hence, the edges out of $\alpha$ and $u^{-1}r\alpha$ have the same direction. By induction, the edges out of all cells in same diagonal as $\alpha$ point in the same direction.
\end{proof}

This lemma will serve as the foundation for a large portion of this paper. An immediate consequence is a nontrivial algorithm to determine if a graph $G_{n,m,2}$ is Hamiltonian. The simplest algorithm merely iterates over all possible directions for all cells, and simulates to check if the result is a cycle. This yields a time complexity of $\mathcal{O}(nm \cdot 2^{4nm})$. We may improve this algorithm by using Lemma~\ref{lemma:rightdown} and only iterating over all possible directions for all diagonals.

\begin{figure}[h]
\centering
\includegraphics[width=0.25\textwidth]{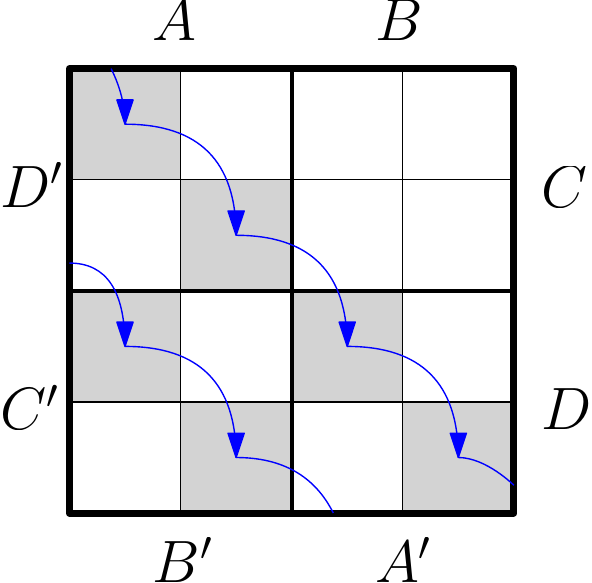}
\caption{One diagonal in $G_{2,2,2}$}\label{fig:diagonal}
\end{figure}

\begin{restatable}{prop}{maxdiagnum}
In the graph $G_{n,m,2}$, there are at most $4\gcd(n,m)$ diagonals.
\end{restatable}

\begin{proof}
For any cell $\alpha$, let $\row(\alpha)$ and $\column(\alpha)$ be the $0$-indexed coordinates of $\alpha$. Then $\row(u^{-1}r\alpha) \equiv \row(\alpha)-1 \pmod{n}$, and $\column(u^{-1}r\alpha) \equiv \column(\alpha)+1 \pmod{m}$. But any diagonal in a torus has length $\lcm(n,m)$. Hence the length of the diagonal in the grid is at least $\lcm(n,m)$ as well. There are $4nm$ elements in the grid, so there are at most $4 \cdot\gcd(n,m)$ diagonals.
\end{proof}

Therefore our second algorithm has complexity $\mathcal{O}(nm\cdot 2^{4\gcd(n,m)})$: since diagonals may be found in $\mathcal{O}(nm)$, iterating over directions for the diagonals dominates the time complexity. We describe a significantly more efficient algorithm after proving several more results and introducing a new perspective of Hamiltonian cycles.

Our algorithm is slow only when $n$ and $m$ are not coprime---that is, $g = \gcd(n,m) > 1$. In this case, we see that many diagonals are congruent, which intuitively may cause redundancy in our algorithm. We call two diagonals $x$ and $x'$ \textit{parallel} if for some $i \geq 1$, $\alpha$ is a cell in diagonal $x$ if and only if $r^i\alpha$ is in diagonal $x'$.

\begin{restatable}{prop}{diagmult}
Let $n$ and $m$ be coprime positive integers, and let $c$ be the number of diagonals in $G_{n,m,2}$. Then for each $g \geq 1$, there are $gc$ diagonals in $G_{gn,gm,2}$.
\end{restatable}

\begin{proof}
If two cells in $G_{gn,gm,2}$ have different column-row differences modulo $g$, they must belong to different diagonals: cells on the same diagonal segment have the same difference, and since all dimensions of the grid are multiples of $g$, two consecutive segments in one diagonal have the same differences modulo $g$.

For this reason, we can fix a remainder $i \in \inter{0}{g}$ and find the number of diagonals in $G_{gn,gm,2}$ with column-row differences equivalent to $i$ modulo $g$. Consider the smaller graph where the $4gmn$ vertices are all cells in $G_{gn,gm,2}$ with column-row difference equivalent to $i$ modulo $g$, and the edges are from each cell $\alpha$ to $u^{-1}r\alpha$. Since every cycle in this graph must contain at least one cell with row index equivalent to $0$ modulo $g$, we may delete all other cells from the graph, merging the in-edge of a deleted cell with the out-edge, without changing the number of cycles. But now the remaining $4mn$ cells may be placed in bijection with the $4mn$ cells in $G_{n,m,2}$, with the mapping
$$(j_1,j_2) \to \left(\floor{j_1}{g},\floor{j_2}{g}\right).$$

Essentially, each $g \times g$ block of the graph $G_{gn,gm,2}$ is shrunk to a single cell in $G_{n,m,2}$. So after applying the mapping to the edges in our constructed graph, we obtain exactly the edges from each cell $\beta \in G_{n,m,2}$ to $u^{-1}r\beta$. Thus the number of diagonals in $G_{gn,gm,2}$ with difference $i$ modulo $g$ is exactly $c$, the number of diagonals in $G_{n,m,2}$.

Summing over all possible values of $i$, we see that in total there are $gc$ diagonals in $G_{gn,gm,2}$.
\end{proof}

As a corollary of the proof of the above proposition, we obtain the following result:

\begin{restatable}{cor}{diagmultcor}
Let $n$ and $m$ be coprime positive integers, and let $g \geq 1$. Then for each $1 \times g$ block of cells at the top left corner of any quadrant of $G_{gn,gm,2}$, the $g$ cells belong, from left to right, to $g$ pairwise-parallel diagonals.
\end{restatable}

\begin{proof}
Consider any two diagonals $x$ and $x'$ containing adjacent cells in the block. Both diagonals correspond to the same diagonal $y$ in $G_{n,m,2}$. Consider any cell $\alpha$ in $x$. As $\alpha$ and $r\alpha$ belong to the same $g \times g$ block of $G_{gn,gm,2}$, they correspond to the same cell in $G_{n,m,2}$. But since  $\alpha$ is in $x$, this cell is in $y$, so $r\alpha$ is in $x'$. The reverse direction is similar.
\end{proof}

We call these $g$ parallel diagonals a \textit{group}. Each block corresponds to one group, and if the diagonals in one group each have length $kgnm$ for some positive integer $k$, then the group corresponds to $k$ blocks.


\subsection{Two-holed Torus Links}\label{subsec:links}

At this point, we understand how the diagonals in a graph $G_{kn,km,2}$ with non-coprime sizes relate to the diagonals in the graph $G_{n,m,2}$ with coprime sizes. To leverage this relation into a result about the Hamiltonicity of the graph, we must convert the problem of determining the Hamiltonicity of our graphs into multiple simpler problems about the connectedness of a two-holed torus link. In some sense the continuous analog of a permutation graph, a \textit{link} is a set of smooth and non-self-intersecting loops, and we are concerned with determining whether certain links embedded on a two-holed torus are \textit{knots}---single, connected loops.

Some two-holed torus links can be described by a four-tuple $(a,b,c,d)$. Imagining the link to reside on a rectangular surface folded into the shape of the two-holed torus, the parameters count the number of times the link passes the left upper boundary, right upper boundary, upper right boundary, and lower right boundary, respectively (in the future, we refer to these boundaries as $A$, $B$, $C$, and $D$---see Figure~\ref{fig:idmap}). Thus there are $a$ points on the left upper boundary which can also be viewed as $a$ points on the right lower boundary, and so forth: due to the identification map of the two-holed torus, there is really only one set of $a$ points, not two, but we can think of them as distinct sets of $a$ points which are connected pairwise ``under the grid." The points on the upper or right boundaries, in clockwise order, are also connected bijectively ``over the grid" to the points on the left or lower boundaries, in counterclockwise order. See Figure~\ref{fig:link} for an example.

\begin{figure}[h]
\centering
\includegraphics[width=0.25\textwidth]{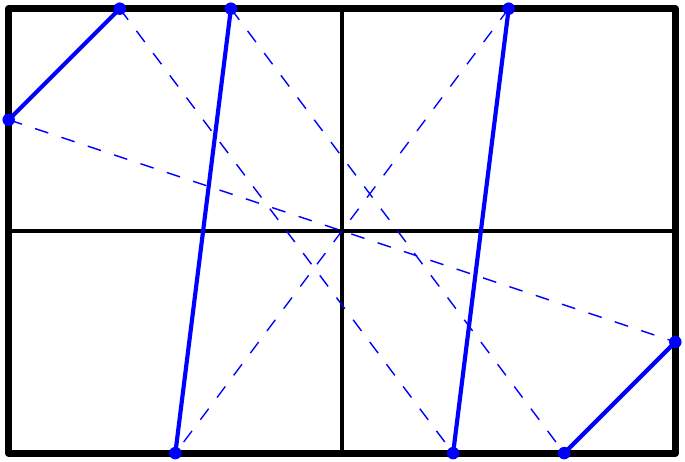}
\caption{A $(2,1,0,1)$ link}\label{fig:link}
\end{figure}

Returning to our graphs $G_{n,m,2}$, every way of orienting the diagonals produces a permutation graph (which may or may not be Hamiltonian) and thus a link (which may or may not be a knot).

\begin{defn}
Let $n$ and $m$ be coprime, and let $G_{n,m,2}$ be a graph with $c$ diagonals. Pick any \textit{orientation string}, which we call $\omega$, consisting of $c$ characters, either $R$ or $U$, each corresponding to the orientation of every cell in a different diagonal. Then define the $\omega$-\textit{link} to be the link with parameters $(a,b,c,d)$ where $a$ and $b$ count the number of cells on boundaries $A$ and $B$ which are oriented up, and $c$ and $d$ count the number of cells on boundaries $C$ and $D$ which are oriented right (see Figure~\ref{fig:orientation}).
\end{defn}

\begin{figure}[h]
	\centering
	\begin{subfigure}[t]{0.25\textwidth}
		\centering
		\includegraphics[width=1\textwidth]{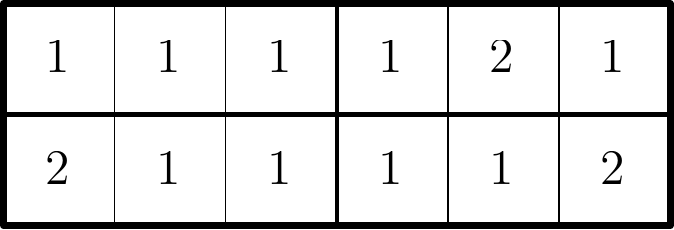}
		\caption{Diagonal of each cell}
	\end{subfigure}
	\begin{subfigure}[t]{0.25\textwidth}
		\centering
		\includegraphics[width=1.0\textwidth]{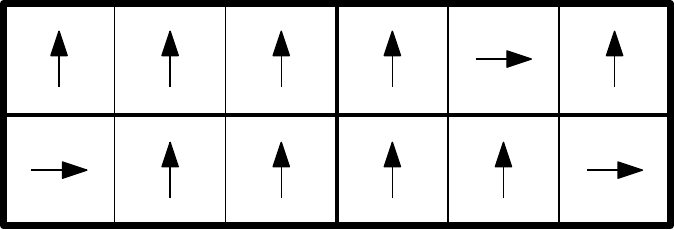}
		\caption{$\omega$-orientation of cells}
	\end{subfigure}

	\begin{subfigure}[t]{0.4\textwidth}
		\centering
		\includegraphics[width=1\textwidth]{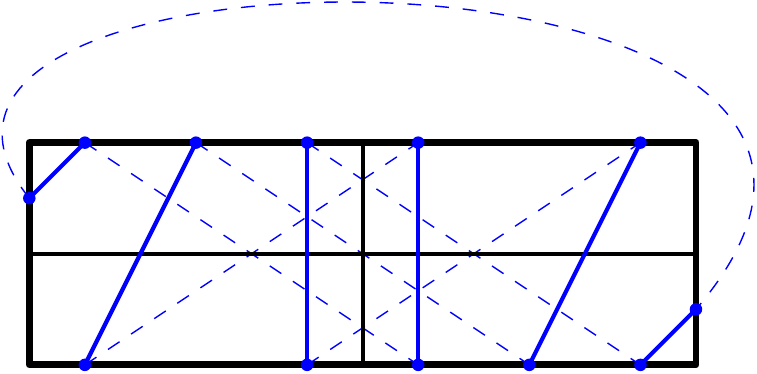}
		\caption{$\omega$-link $(3,2,0,1)$}
	\end{subfigure}
	\caption{Constructing the $\omega$-link from $G_{1,3,2}$, with $\omega=UR$}\label{fig:orientation}
\end{figure}

\begin{restatable}{theorem}{hamlink}
Let $n$ and $m$ be positive integers. Pick any orientation string $\omega$ for $G_{n,m,2}$. Then $\omega$ constructs a Hamiltonian cycle if and only if the $\omega$-link is a knot.
\end{restatable}

\begin{proof}
Suppose we construct a new graph $G'$ with the same vertices as $G_{n,m,2}$ and only the edges which agree with the orientations of their cells---the up edges out of cells oriented up, and the right edges out of cells oriented right. Then, as a consequence of the proof of Lemma~\ref{lemma:rightdown}, $G'$ is a permutation graph. Thus $\omega$ constructs a Hamiltonian cycle if and only if $G'$ is a single cycle rather than multiple.

If $(a,b,c,d)$ is the $\omega$-link, then there are by definition $a$ edges in $G'$ which cross boundary $A$, and so forth. Hence there are essentially $a+b+c+d$ points on the upper or right boundaries of the grid which are glued (``underneath" the grid) to $a+b+c+d$ points on the lower or left boundaries; therefore to show that $G'$ is homeomorphic to the $\omega$-link, we only need to show that the ``over-grid" connections in $G'$ between pairs of points are the same as the ``over-grid" connections in the $\omega$-link between pairs of points.

But for there to be an over-grid connection between two points in $G'$, the point on the upper or right boundary must be above and to the right of the point on the lower or left boundary. This completely determines the connections. The farthest counterclockwise point on the upper or right boundary must connect to the farthest clockwise point on the lower or left boundary, since otherwise any path from the latter point to a point on the upper or right boundary would intersect the path out of the former point.

By induction, all points in $G'$ must connect exactly as the corresponding points connect in the $\omega$-link. Hence $\omega$ constructs a Hamiltonian cycle exactly when the $\omega$-link is a knot.
\end{proof}



We can now describe a pseudo-polynomial time algorithm for determining when a graph $G_{n,m,2}$ is Hamiltonian.

\begin{restatable}{prop}{swapparallel}
Let $n$ and $m$ be positive integers. Let $\omega$ be any orientation string for $G_{n,m,2}$. Let $\omega'$ be the orientation string constructed from $\omega$ but with the orientations of two parallel diagonals swapped. Then $\omega$ constructs a Hamiltonian cycle if and only if $\omega'$ constructs a Hamiltonian cycle.
\end{restatable}

\begin{proof}
Let $g = \gcd(n,m)$. Since diagonals $d_i$ and $d_j$ are both similar (under a scaling factor of $g$) to the same diagonal in $G_{n/g,m/g,2}$, they contain the same number of cells on boundary $A$, the same number of cells on boundary $B$, and so forth. Hence swapping their orientations does not change any of the parameters $(a,b,c,d)$ of the link constructed, so it does not affect the Hamiltonicity of the constructed permutation graph.
\end{proof}


Thus for each graph $G_{n,m,2}$ with $g = \gcd(m,n)$, our algorithm does not need to check all $2^g$ ways of orienting a group of $g$ parallel diagonals. Two ways are equivalent if the number of diagonals in the group which are oriented up is the same. Hence, for each group only $g+1$ ways need to be checked---when $i$ of them are oriented up for $0 \leq i \leq g$. As there are at most $4$ groups, one per quadrant, only $(g+1)^4 = \mathcal{O}(g^4)$ orientation strings must be checked. This leads to an overall running time for our third algorithm of $\mathcal{O}(mng^4)$.

Using several later results, it is possible to improve this further to $\mathcal{O}((m+n)g^3)$: it turns out that there are at most $3$ groups, and that we can check if a link is a knot in linear time in the sum of parameters.

\subsection{Periodicity of Hamiltonicity}\label{subsec:periodicity}

With the torus link equivalence, we can now show that for a fixed $n$, as $m$ varies, the property of Hamiltonicity of the graphs $G_{n,m,2}$ with $n$ and $m$ coprime is periodic.

\begin{restatable}{prop}{hamlinkperiod}
Let $n$ and $m$ be coprime positive integers. For the graph $G_{n,m,2}$, pick any orientation string $\omega$, and let the $\omega$-link be $(a,b,c,d)$. Then if there are $kmn$ cells oriented up for some integer $k$, the graph $G_{n,m+12n,2}$ has $\omega$-link $(a+3kn,b+3kn,c,d)$.
\end{restatable}

\begin{proof}
We can view $G_{n,m+12n,2}$ as the result of adding $12n$ columns to the right side of each quadrant in $G_{n,m,2}$. Consider any diagonal which enters the new columns at row $i$. Then the column-minus-row index of that diagonal segment is $m-i$ or $2m+12n-i$. Let us assume the index is $m-i$, as the other case is almost identical. The subsequent diagonal segments are $2m+12n-i+2n$, $m-i+4n$, $2m+12n-i+6n$, $m-i+8n$, $2m+12n-i+10n$, and $m-i+12n$, so the last element still in the new columns is at row $i-1$, and thus the first element in the original graph is at row $i$, where it would have been without the extra $24n$ columns. See Figure~\ref{fig:periodicity} for a diagram.

\begin{figure}[h]
\centering
\includegraphics[width=1\textwidth]{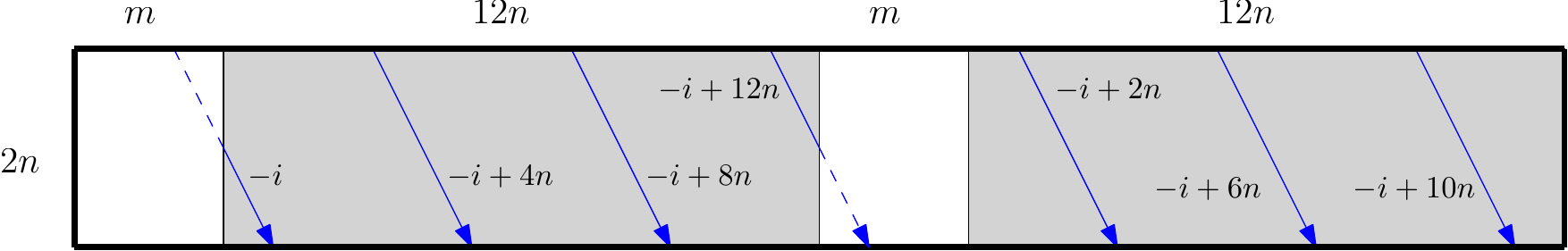}
\caption{The diagonal segments introduced by adding $24n$ columns to $G_{n,m,2}$ (indexed by column-row difference within grey box)}\label{fig:periodicity}
\end{figure}

It follows that the subset of the diagonal contained in the original grid is equal to the complete diagonal in the graph $G_{n,m,2}$. Each passage of each up-oriented diagonal through the extra columns contributes $3$ to $a$ and $b$, but does not affect $c$ or $d$. But if there are $kmn$ cells oriented up in $G_{n,m,2}$, if we reduce to the standard torus grid there are $k$ Hamiltonian cycles, so we hit the left boundary $kn$ times. Lifting to $G_{n,m,2}$, we enter the new columns $kn$ times, so in total we add $3kn$ to $a$ and to $b$.
\end{proof}


It remains to show that links are periodic with a certain periodicity.

\begin{restatable}{prop}{linkperiod}
Let $(a,b,c,d)$ be a link. Let $t = -a+b+2c+2d$, and assume that $a,b > t \geq c+d$. Then $(a-t,b-t,c,d)$ has the same number of loops as $(a,b,c,d)$.
\end{restatable}

\begin{proof}
The segments of the link $(a,b,c,d)$ which connect boundary points pairwise can be numbered from $0$ to $a+b+c+d-1$. If we draw an edge from each segment to the subsequent segment in the link, we can describe the edges in intervals as follows:
\begin{itemize}
\item $\inter{0}{a} \longrightarrow \inter{b+c+d}{a+b+c+d}$
\item $\inter{a}{a+b} \longrightarrow \inter{c+d}{b+c+d}$
\item $\inter{a+b}{a+b+c} \longrightarrow \inter{d}{c+d}$
\item $\inter{a+b+c}{a+b+c+d} \longrightarrow \inter{0}{d}$
\end{itemize}

This notation means that $0 \longrightarrow b+c+d$, and $1 \longrightarrow b+c+d+1$ (if $a>1$), and so forth.

We wish to show that this graph can be transformed into the corresponding graph for $(a-t,b-t,c,d)$ without changing the number of cycles.

First, we delete the interval of vertices $\inter{a-t}{a}$ and the interval $\inter{a+b+c+d-t}{a+b+c+d}$; the edge into any deleted vertex points to the vertex after the deleted vertex. Then we renumber (shift) the interval $\inter{a}{a+b+c+d-t}$ to $\inter{a-t}{a+b+c+d-2t}$.

Now we must show that starting with any vertex in the new graph, if we transform back to the old graph, apply the function, and transform to the new graph, we obtain the desired edge.

Suppose we start in the interval $\inter{0}{a-t}$ in the new graph. This is equal to $\inter{0}{a-t}$ in the old graph. Applying the edges, it leads to $\inter{b+c+d}{a+b+c+d-t}$ in the old graph. This must be shifted, since $a \leq b+c+d$, so in the new graph we have $\inter{b+c+d-t}{a+b+c+d-2t}$ or equivalently $\inter{b-t+c+d}{a-t+b-t+c+d}$ as desired.

Suppose we start in the interval $\inter{a-t}{a-t+b-t}$ in the new graph. This must be unshifted to $\inter{a}{a+b-t}$ in the old graph. Applying the edges, it leads to $\inter{c+d}{b+c+d-t}$ in the old graph. As $b+c+d-t = a-c-d < a$, this interval is equivalent to $\inter{c+d}{b+c+d-t}$ in the new graph.

Suppose we start in the interval $\inter{a-t+b-t}{a-t+b-t+c}$ in the new graph. This must be unshifted to $\inter{a+b-t}{a+b+c-t}$ in the old graph. Since $a+b+c-t = 2a-c-2d < 2a-c-d \leq a+b$, the interval is contained in $\inter{a}{a+b}$. Thus applying the edges we obtain $\inter{b+c+d-t}{b+2c+d-t} = \inter{a-c-d}{a-d}$ in the old graph. As this interval has been removed, we apply the edges again to obtain $\inter{a+b}{a+b+c}$. This interval has been removed as well, so we apply the edges once more to obtain $\inter{d}{c+d}$ in the old graph. We do not need to shift, so the interval is equivalent to $\inter{d}{c+d}$ in the new graph, as desired.

Finally, suppose we start in the interval $\inter{a-t+b-t+c}{a-t+b-t+c+d}$ in the new graph. This must be unshifted to $\inter{a+b+c-t}{a+b+c+d-t}$, which is contained in $\inter{a}{a+b}$. Applying the edges, we obtain $\inter{b+2c+d-t}{b+2c+2d-t} = \inter{a-d}{a}$. These vertices have been removed, so we apply the edges to obtain $\inter{a+b+c}{a+b+c+d}$. These vertices have also been removed, so we apply the edges again, to get the interval $\inter{0}{d}$ as required.
\end{proof}


For links constructed by some graph and some orientation string, the periodicity $t$ does relate to the size of the graph.

\begin{restatable}{prop}{coolproof}
Let $n$ and $m$ be coprime positive integers. For the graph $G_{n,m,2}$, pick any orientation string $\omega$, and let the $\omega$-link be $(a,b,c,d)$. Then if there are $kmn$ cells oriented up for some integer $k$, we have $-a + b + 2c + 2d = (4-k)n$.
\end{restatable}

\begin{proof}
Let $w_r$ and $w_u$ denote the number of cells on boundary $A$ oriented right and up respectively. Define $x_r$, $x_u$, $y_r$, $y_u$, $z_r$, and $z_u$ correspondingly for boundaries $B$, $C$, and $D$.

Consider the cyclic sequences of cells---one sequence for each diagonal oriented up---ordered so that cell $u^{-1}r\alpha$ follows cell $\alpha$. The sum over all sequences of the signed differences of column indices of pairs of cells adjacent in a sequence is $0$. Each of the $kmn$ cells contributes $+1$. Additionally, every time a boundary is crossed between adjacent cells, there is an additional displacement of some multiple of $m$. Adding all contributions, we obtain the equation $$kmn+x_um-w_um-2y_um-2z_um = 0.$$
Thus $$kn+b-a-2(n-c)-2(n-d)=0$$ or equivalently $$-a+b+2c+2d=(4-k)n.$$
\end{proof}

We now simply combine the above propositions.

\begin{restatable}{theorem}{hamperiod}
Let $n$ and $m$ be coprime positive integers. Then $G_{n,m,2}$ is Hamiltonian if and only if $G_{n,m+12n,2}$ is Hamiltonian.
\end{restatable}

\begin{proof}
Pick any orientation string $\omega$. Let $k$ be the integer such that $kmn$ cells are oriented up, and let the induced $\omega$-link be $(a,b,c,d)$. As $c+d$ is at most the number of cells on the right boundary of their quadrants which are oriented right, we have $c+d \leq (4-k)n = -a + b + 2c + 2d$. Thus we know that $(a,b,c,d)$ has the same structure as $(a+(4-k)pn,b+(4-k)pn,c,d)$ for any positive integer $p$. But for $k \in \{1,2,3\}$ we have $(4-k)n \mid 3kn$, so in fact $(a,b,c,d)$ has the same structure as $(a+3kn,b+3kn,c,d)$. The latter link is the $\omega$-link of $G_{n,m+12n,2}$, so $\omega$ produces a Hamiltonian cycle in $G_{n,m,2}$ if and only if it produces a Hamiltonian cycle in $G_{n,m+12n,2}$.
\end{proof}

Through a proof almost identical to that described in this section, it is also possible to show that for any fixed $g$ and fixed $n$, the graphs $G_{n,m,2}$ with $\gcd(n,m)=g$ are periodic in Hamiltonicity.

\begin{restatable}{theorem}{genhamperiod}
Let $n$ and $m$ be positive integers, and let $g = \gcd(n,m)$. Then $G_{n,m,2}$ is Hamiltonian if and only if $G_{n,m+4(4g)!n,2}$ is Hamiltonian.
\end{restatable}

\begin{proof}
This proof is similar to the proof for $g=1$.
\end{proof}





\section{Counting Grid Diagonals}\label{sec:counting}

In this section we attempt to classify our graphs by the number of diagonals. We find two sets of reductions from larger to smaller graphs that preserve the number of diagonals, and use these to develop a polynomial time algorithm to find the number of diagonals.

\subsection{Boundary-Crossing String Reductions}\label{subsec:reductions}

In our first set of reductions, we find for each graph a \textit{boundary-crossing string} which corresponds to a permutation, and use properties of permutations to reduce longer strings (associated with larger graphs) to smaller strings.

\begin{defn}
Let $n,m \geq 1$. Define $\diag(n,m)$ to be the number of diagonals in $G_{n,m,2}$.
\end{defn}

To find patterns in the number of diagonals in $G_{n,m,2}$, we first consider how the diagonal crosses boundaries on the standard torus grid of size $n \times m$, and lift results there back to $G_{n,m,2}$. Since $\diag(cn,cm) = c \cdot \diag(n,m)$ for any $n$, $m$, and $c$, we solely consider the case when $n$ and $m$ are coprime.

\begin{defn}
Let $n,m>1$ with $\gcd(n,m)=1$. Consider the diagonal on the $n \times m$ (standard) torus grid, starting at the top-left corner. Let $s_{n,m}$ be the string of characters indicating each time when the diagonal passes through a boundary of the grid---$\dd$ for the lower boundary and $\rr$ for the right boundary. Furthermore, define the \textit{boundary-crossing string} for $G_{n,m,2}$ to be $t_{n,m} = \dd s_{n,m} \dd^{-1}$. This adjustment will simplify later proofs.
\end{defn}

When $n$ and $m$ are coprime, every diagonal in $G_{n,m,2}$ contains the top-left corner of at least one quadrant of the base grid. Hence, if we consider each diagonal to start at one such corner, $s_{n,m}$ describes how the diagonal crosses quadrants. Viewing $\dd$ and $\rr$ as permutations of the four quadrants, we have that $\diag(n,m)$, where $n,m>1$ and $\gcd(n,m)=1$, is equal to the number of cycles in the permutation $s_{n,m}$ (or equivalently, in the permutation $t_{n,m}$)---for instance, if the permutation has one cycle, then the diagonal starting at the top-left corner of one quadrant will visit all four corners before returning to the starting point, so there will be only one diagonal, of length $4nm$. We use the notation $\cyc(\phi)$ to denote the number of cycles in permutation $\phi$, and say that $\phi \sim \pi$ if permutations $\phi$ and $\pi$ are equal up to renaming of elements.

Now we describe two procedures for generating $s_{n,m}$ or $t_{n,m}$.

\begin{restatable}{lemma}{intervals}\label{lemma:intervals}
Let $n,m>1$ with $\gcd(n,m)=1$. The string $s_{n,m}$ can be constructed by the following procedure: traverse the sequence $n, 2n, \dots, (m-1)n, mn$. For each element $jn$, add $\rr^{i} \dd$, where $i = \left \lfloor \frac{jn}{m} \right \rfloor - \left \lfloor \frac{(j-1)n}{m} \right \rfloor$, to the end of the string.
\end{restatable}

\begin{proof}
Suppose the $n \times m$ torus grid cells are numbered from $0$ to $nm-1$ by the order of traversal on the diagonal. Then at cell $i$, the bottom boundary has been crossed $\left \lfloor \frac{i}{n} \right \rfloor$ times, and the right boundary has been crossed $\left \lfloor \frac{i}{m} \right \rfloor$ times.

The bottom boundary is crossed immediately before the locations $jn$ for $0 < j \leq nm = 0$. Thus we add $\dd$ to the string at these locations. Between $(j-1)n$ and $jn$, the right boundary is crossed  $\left \lfloor \frac{jn}{m} \right \rfloor - \left \lfloor \frac{(j-1)n}{m} \right \rfloor$ times, so we add this many powers of $\rr$.
\end{proof}


\begin{restatable}{lemma}{powers}\label{lemma:powers}
Let $n,m>1$ with $\gcd(n,m)=1$. Suppose $m = nk+p$ where $0 < p < n$. The string $t_{n,m}$ can be constructed by the following procedure: traverse the generating sequence $0,-m,-2m,\dots,-(n-1)m$. For each element $im$, if $(im \mod{n}) < p$, add $\dd^{k+1} \rr$ to the end of the string. Otherwise add $\dd^k \rr$ to the end of the string.
\end{restatable}

\begin{proof}
Let $i$ be an integer with $0 \leq i < n$. Then the right boundary is crossed immediately before cell $(i+1)m$. In the range $[im,(i+1)m)$, the bottom boundary is crossed $k+1$ times if $$im \leq jn < im + p$$ for some integer $j$ and $k$ times otherwise. This condition is equivalent to $$-im \mod{n} < p.$$ Note that cell $0$ is treated correctly; we wish to account for the down-crossing at cell $0$ at the beginning of $t_{n,m}$ but not the end, and since for $j=0$ we include $0$ in the range, and for $j=n-1$ we do not include $nm$, this is satisfied.
\end{proof}


Now we prove a useful lemma for manipulating permutation sequences.

\begin{restatable}{lemma}{swap}~\label{lemma:swap}
Let $\phi$ and $\pi$ be permutations of an arbitrary set, and let $m,n>0$. Then $$\prod_{i=1}^m \phi^{\ceil{in}{m} - \ceil{(i-1)n}{m}} \pi = \prod_{j=1}^n \phi \pi^{\floor{jm}{n}-\floor{(j-1)m}{n}}.$$
\end{restatable}

\begin{proof}
There are two cases. First suppose $n \leq m$. Then for any integer $i$ we have $\frac{in}{m} - \frac{(i-1)n}{m} \leq 1$, so $0 \leq \ceil{in}{m} - \ceil{(i-1)n}{m} \leq 1$.

Suppose $\ceil{in}{m} - \ceil{(i-1)n}{m} = 1$ where $1 \leq i \leq m$. Then for some $j$ with $0 \leq j \leq n-1$, we have $(i-1)n \leq jm < in$. For each $j$ with $0 \leq j \leq n-1$, there is exactly one index $i(j) = 1 + \floor{jm}{n}$ satisfying that inequality. Hence, $\phi$ occurs only at terms $i(0),i(1),\dots,i(n-1)$, so we have
\begin{align*}
\prod_{i=1}^m \phi^{\ceil{in}{m} - \ceil{(i-1)n}{m}} \pi
&= \prod_{j=0}^{n-1} \phi \pi^{i(j+1) - i(j)}\\
&= \prod_{j=0}^{n-1} \phi \pi^{1 + \floor{(j+1)m}{n}- 1 - \floor{jm}{n}} \\
&= \prod_{j=1}^n \phi \pi^{\floor{jm}{n} - \floor{(j-1)m}{n}}.
\end{align*}

For the second case, suppose $n \geq m$, and suppose $\floor{in}{m} - \floor{(i-1)n}{m} = 1$ where $1 \leq i \leq m$. Then for some $j$ with $1 \leq j \leq n$, we have $(i-1)n < jm \leq in$. But for each $j$ with $1 \leq j \leq n$, there is exactly one index $i(j) = \ceil{jm}{n}$ satisfying that inequality. Hence, $\pi$ occurs only at terms $i(1), \dots, i(n)$, so we have
\begin{align*}
\prod_{i=1}^m \phi \pi^{\floor{in}{m} - \floor{(i-1)n}{m}}
&= \prod_{j=1}^n \phi^{i(j) - i(j-1)} \pi \\
&= \prod_{j=1}^n \phi^{\ceil{jm}{n} - \ceil{(j-1)m}{n}} \pi.
\end{align*}
\end{proof}


Using Lemma~\ref{lemma:swap}, we can convert our expression for $t_{n,m}$ into a product of terms with the ceiling function in the exponent.

\begin{restatable}{cor}{pceil}~\label{cor:ceil}
Let $n,m>1$ with $\gcd(n,m)=1$. Then $$t_{n,m} = \prod_{i=1}^n \dd^{\ceil{im}{n}-\ceil{(i-1)m}{n}} \rr.$$
\end{restatable}

\begin{proof}
Applying Lemma~\ref{lemma:swap} to the formula in Lemma~\ref{lemma:intervals}, we have
\begin{align*}
t_{n,m}
&= \dd s_{n,m} \dd^{-1} \\
&= \dd \left ( \prod_{j=1}^m \rr^{\floor{jn}{m} - \floor{(j-1)n}{m}} \dd \right ) \dd^{-1} \\
&= \prod_{j=1}^m \dd \rr^{\floor{jn}{m} - \floor{(j-1)n}{m}} \\
&= \prod_{i=1}^n \dd^{\ceil{im}{n} - \ceil{(i-1)m}{n}} \rr.
\end{align*}
\end{proof}


With the aid of the above lemmas, we describe and prove several symmetries in the number of diagonals on different graphs, using the properties of the permutations $\dd$ and $\rr$.

\begin{restatable}{prop}{modfour}\label{prop:modfour}
Let $n,m>1$ with $\gcd(n,m)=1$. Then the number of diagonals in $G_{n,m,2}$ is equal to the number of diagonals in $G_{n,4n+m,2}$.
\end{restatable}

\begin{proof}
Applying Corollary~\ref{cor:ceil} and using the equation $\dd^4 = id$, we have
\begin{align*}
t_{n,4n+m}
&= \prod_{i=1}^n \dd^{\ceil{i(4n+m)}{n}-\ceil{(i-1)(4n+m)}{n}} \rr \\
&= \prod_{i=1}^n \dd^{4 +\ceil{im}{n}-\ceil{(i-1)m}{n}} \rr \\
&= \prod_{i=1}^n \dd^{\ceil{im}{n}-\ceil{(i-1)m}{n}} \rr \\
&= t_{n,m}.
\end{align*}
\end{proof}

\begin{restatable}{prop}{reversemod}\label{prop:reversemod}
Let $n,m>1$ with $\gcd(n,m)=1$ and $4n>m>1$. Then the number of diagonals in $G_{n,m,2}$ is equal to the number of diagonals in $G_{n,4n-m,2}$.
\end{restatable}

\begin{proof}
First we show that to obtain $t_{n,4n-m}$ from $t_{n,m}$ we reverse the string, replace each substring $\dd^i$ with $\dd^{4-i}$, and move one power of $\rr$ from the beginning of the string to the end. After that, we show that this procedure preserves the number of cycles.

Once again, we use the construction of Lemma~\ref{lemma:powers}.

Let $m = nk+p$ where $0 < p < n$, which means that $4n-m = n(4-k-1)+n-p$. Let $i$ be an index into the $(n,m)$ generating sequence with $0 \leq i < n$.

Suppose the substring of $t_{n,m}$ contributed by term $i$ is $\dd^{k+1} \rr$. Then $$0 \leq -ip \mod{n} < p.$$ Hence, we have $$n-p \leq -(i+1)p \mod{n} < n$$ or equivalently $$n-p \leq -(n-i-1)(n-p) \mod{n} < n.$$ Thus the term $n-1-i$ in the $(n,4n-m)$ generating sequence contributes $\dd^{4-k-1}\rr$ to $t_{n,4n-m}$.

Conversely, suppose term $i$ in the $(n,m)$ generating sequence contributes $\dd^k\rr$ to $t_{n,m}$. Then $$p \leq -ip \mod{n} < n,$$ so $$0 \leq -(i+1)p \mod{n} < n-p$$ or equivalently, $$0 \leq -(n-i-1)(n-p) \mod{n} < n-p.$$ Thus the term $n-1-i$ in the $(n,4n-m)$ generating sequence contributes $\dd^{4-k}\rr$ to $t_{n,4n-m}$.

It follows that after the reversal of $t_{n,m}$ and replacement of each $\dd^i$ by $\dd^{4-i}$, the powers of $\dd$ are the same as in $t_{n,4n-m}$, but we must change every occurrence of $\rr \dd^i$ to $\dd^i \rr$. This is accomplished by moving the first $\rr$ to the end of the string.

This proves the first piece of the argument. For the second piece, note that a permutation $\phi$ has the same number of cycles as $\phi^{-1}$ and $\pi \phi \pi^{-1}$ for any permutation $\pi$. Furthermore, note that $(1\;2)(3\;4)\dd(3\;4)(1\;2) = \dd$ whereas $(1\;2)(3\;4)\rr(3\;4)(1\;2) = \rr^{-1}$. Then $$t_{n,4n-m} = \rr^{-1} (1\;2)(3\;4)t_{n,m}^{-1}  (3\;4)(1\;2) \rr$$ since the inversion reverses the string and inverts every $\dd$ and every $\rr$, and the inner two conjugations invert every $\rr$ again, and the outer conjugation moves an $\rr$ from the beginning of the string to the end. So with this equality, we can produce $t_{n,4n-m}$ from $t_{n,m}$ by applying conjugates and taking inverses, so the number of cycles in $t_{n,m}$ is equal to the number of cycles in $t_{n,m}$.
\end{proof}


Now we show that all graphs $G_{n,m,2}$ can be reduced to a few base case graphs while preserving the number of diagonals. We fix $n$ and $m$ as coprime positive integers, and let $q_0$, $q_1$, $q_2$, $r_0$, $r_1$, and $r_2$ be defined so that $m = q_0n+r_0$ and $n = q_1r_0 + r_1$ and $r_0 = q_2r_1+r_2$ with $0 \leq r_2 \leq r_1 < r_0 < n$, and $r_2 < r_1$ if $r_1 > 0$.

As several of the proofs start with the same algebraic manipulations, we collect those steps into the following lemma.

\begin{restatable}{lemma}{twoswap}\label{lemma:twoswap}
Let $n,m \geq 1$. Then $$t_{n,m} = \prod_{j=1}^{r_0} \dd (\dd^{q_0} \rr)^{q_1+\floor{jr_1}{r_0}-\floor{(j-1)r_1}{r_0}}.$$
\end{restatable}

\begin{proof}
Applying Lemma~\ref{lemma:swap} to the formula in Corollary~\ref{cor:ceil}, we obtain
\begin{align*}
t_{n,m}
&= \prod_{i=1}^n \dd^{\ceil{im}{n}-\ceil{(i-1)m}{n}}\rr \\
&= \prod_{i=1}^n \dd^{q_0 + \ceil{ir_0}{n}-\ceil{(i-1)r_0}{n}}\rr \\
&= \prod_{i=1}^n \dd^{\ceil{ir_0}{n}-\ceil{(i-1)r_0}{n}} \dd^{q_0} \rr \\
&= \prod_{j=1}^{r_0} \dd (\dd^{q_0}\rr)^{\floor{jn}{r_0}-\floor{(j-1)n}{r_0}} \\
&= \prod_{j=1}^{r_0} \dd (\dd^{q_0}\rr)^{q_1+\floor{jr_1}{r_0}-\floor{(j-1)r_1}{r_0}}
\end{align*}
as desired.
\end{proof}


The following proposition is a compilation of all the reductions.

\begin{restatable}{prop}{reductions}
We have:
\begin{singlespacing}
\begin{itemize}
\item If $q_0 \geq 4$, then $\diag(n,m) = \diag(n,(q_0-4)n+r_0)$.
\item If $q_0 = 3$, then $\diag(n,m) = \diag(n,n-r_0)$.
\item If $q_0 = 2$, then $\diag(n,m) = \diag(n,2n-r_0)$.
\item If $q_0 = 1$ and $q_1 \geq 4$, then $\diag(n,m) = \diag((q_1-3)r_0+r_1,(q_1-2)r_0+r_1)$.
\item If $q_0=1$, $q_1 = 3$ and $r_1>0$, then $\diag(n,m) = \diag(r_1,r_0+r_1)$.
\item If $q_0=1$, $q_1=2$ and $r_1>0$, then $\diag(n,m) = \diag(r_1,r_0-r_1)$.
\item If $q_0=1$, $q_1=1$, $q_2$ is even, and $r_1>r_2>0$, then $\diag(n,m) = \diag(r_1+r_2,r_1+2r_2)$.
\item If $q_0=1$, $q_1=1$, $q_2$ is even, and $r_1>r_2=0$, then $\diag(n,m) = \diag(1,1)$.
\item If $q_0=1$, $q_1=1$, $q_2$ is odd, and $r_1>r_2>0$, then $\diag(n,m) = \diag(r_2,r_1+2r_2)$.
\item If $q_0=1$, $q_1=1$, $q_2$ is odd, and $r_1>r_2=0$, then $\diag(n,m) = \diag(2,3)$.
\end{itemize}
\end{singlespacing}
\end{restatable}

\begin{proof}
The first three reductions follow immediately from Proposition~\ref{prop:modfour} and Proposition~\ref{prop:reversemod}.

Suppose $q_0 = 1$ and $q_1 \geq 4$. Starting with the formula from Lemma~\ref{lemma:twoswap}, we have
\begin{align*}
t_{n,m}
&= \prod_{j=1}^{r_0} \dd (\dd\rr)^{q_1+\floor{jr_1}{r_0}-\floor{(j-1)r_1}{r_0}} \\
&= \prod_{j=1}^{r_0} \dd (\dd\rr)^{q_1-3+\floor{jr_1}{r_0}-\floor{(j-1)r_1}{r_0}} \\
&= \prod_{j=1}^{r_0} \dd (\dd\rr)^{\floor{j(n-3r_0)}{r_0}-\floor{(j-1)(n-3r_0)}{r_0}} \\
&= \prod_{i=1}^{n-3r_0} \dd^{\ceil{ir_0}{n-3r_0}-\ceil{(i-1)r_0}{n-3r_0}} \dd \rr \\
&= \prod_{i=1}^{n-3r_0} \dd^{\ceil{i(n-2r_0)}{n-3r_0}-\ceil{(i-1)(n-2r_0)}{n-3r_0}} \rr \\
&= t_{n-3r_0,n-2r_0}.
\end{align*}

Suppose $q_0=1$ and $q_1 = 3$ and $r_1>0$. We have
\begin{align*}
t_{n,m}
&= \prod_{j=1}^{r_0} \dd (\dd\rr)^{3+\floor{jr_1}{r_0}-\floor{(j-1)r_1}{r_0}} \\
&= \prod_{j=1}^{r_0} \dd (\dd\rr)^{\floor{jr_1}{r_0}-\floor{(j-1)r_1}{r_0}} \\
&= \prod_{i=1}^{r_1} \dd^{\ceil{ir_0}{r_1}-\ceil{(i-1)r_0}{r_1}} \dd \rr \\
&= \prod_{i=1}^{r_1} \dd^{\ceil{i(r_0+r_1)}{r_1}-\ceil{(i-1)(r_0+r_1)}{r_1}} \rr \\
&= t_{r_1,r_0+r_1}.
\end{align*}

Suppose $q_0=1$, $q_1=2$, and $r_1>0$. Then we have
\begin{align*}
t_{n,m}
&= \prod_{j=1}^{r_0} \dd (\dd\rr)^{2+\floor{jr_1}{r_0}-\floor{(j-1)r_1}{r_0}} \\
&= \prod_{j=1}^{r_0} (\dd (\dd\rr)^2) (\dd\rr)^{\floor{jr_1}{r_0}-\floor{(j-1)r_1}{r_0}} \\
&= \prod_{i=1}^{r_1} (\dd (\dd\rr)^2)^{\ceil{ir_0}{r_1}-\ceil{(i-1)r_0}{r_1}} \dd\rr \\
&= \prod_{i=1}^{r_1} (\dd (\dd\rr)^2)^{\ceil{i(r_0-r_1)}{r_1}-\ceil{(i-1)(r_0-r_1)}{r_1}} (\dd(\dd\rr)^3) \\
&\sim \dd^{-1} \left ( \prod_{i=1}^{r_1} (\dd (\dd\rr)^2)^{\ceil{i(r_0-r_1)}{r_1}-\ceil{(i-1)(r_0-r_1)}{r_1}} (\dd(\dd\rr)^3) \right ) \dd \\
&= \prod_{i=1}^{r_1} \rr^{-\ceil{i(r_0-r_1)}{r_1}+\ceil{(i-1)(r_0-r_1)}{r_1}} \dd \\
&\sim \prod_{i=1}^{r_1} \dd^{-\ceil{i(r_0-r_1)}{r_1}+\ceil{(i-1)(r_0-r_1)}{r_1}} \rr \\
&\sim \prod_{i=1}^{r_1} \dd^{\ceil{i(r_0-r_1)}{r_1}-\ceil{(i-1)(r_0-r_1)}{r_1}} \rr \\
&= t_{r_1,r_0-r_1}.
\end{align*}

Suppose $q_0=1$ and $q_1=1$, and $q_2$ is even, and $r_1>r_2>0$. We have
\begin{align*}
t_{n,m}
&= \prod_{j=1}^{r_0} \dd (\dd\rr)^{1+\floor{jr_1}{r_0}-\floor{(j-1)r_1}{r_0}} \\
&= \prod_{j=1}^{r_0} \dd^2 \rr (\dd\rr)^{\floor{jr_1}{r_0}-\floor{(j-1)r_1}{r_0}} \\
&= \prod_{j=1}^{r_1} (\dd^2 \rr)^{\ceil{jr_0}{r_1}-\ceil{(j-1)r_0}{r_1}} \dd \rr\\
&= \prod_{j=1}^{r_1} (\dd^2 \rr)^{q_2 + \ceil{jr_2}{r_1}-\ceil{(j-1)r_2}{r_1}} \dd \rr\\
&= \prod_{j=1}^{r_1} (\dd^2 \rr)^{\ceil{jr_2}{r_1}-\ceil{(j-1)r_2}{r_1}} \dd \rr\\
&= \prod_{j=1}^{r_2} \dd^2 \rr (\dd \rr)^{\floor{jr_1}{r_2}-\floor{(j-1)r_1}{r_2}}\\
&= \prod_{j=1}^{r_2} \dd (\dd \rr)^{\floor{j(r_1+r_2)}{r_2}-\floor{(j-1)(r_1+r_2)}{r_2}}\\
&= \prod_{j=1}^{r_1+r_2} \dd^{\ceil{jr_2}{r_1+r_2}-\ceil{(j-1)(r_2}{r_1+r_2}} \dd \rr\\
&= \prod_{j=1}^{r_1+r_2} \dd^{\ceil{j(r_1+2r_2)}{r_1+r_2}-\ceil{(j-1)(r_1+2r_2)}{r_1+r_2}} \rr\\
&= t_{r_1+r_2,r_1+2r_2}.
\end{align*}

Suppose $q_0=1$ and $q_1=1$, and $q_2$ is even, and $r_1>r_2=0$. Note that if $r_2=0$ we must have $r_1 = 1$. Then we have
\begin{align*}
t_{n,m}
&= \prod_{j=1}^{r_0} \dd (\dd\rr)^{1+\floor{jr_1}{r_0}-\floor{(j-1)r_1}{r_0}} \\
&= \prod_{j=1}^{r_0} \dd^2 \rr (\dd\rr)^{\floor{jr_1}{r_0}-\floor{(j-1)r_1}{r_0}} \\
&= \prod_{j=1}^{r_1} (\dd^2 \rr)^{\ceil{jr_0}{r_1}-\ceil{(j-1)r_0}{r_1}} \dd \rr\\
&= \prod_{j=1}^{r_1} (\dd^2 \rr)^{q_2 + \ceil{jr_2}{r_1}-\ceil{(j-1)r_2}{r_1}} \dd \rr\\
&= \prod_{j=1}^{r_1} \dd \rr\\
&= \dd \rr \\
&= t_{1,1}.
\end{align*}

Suppose $q_0=1$ and $q_1=1$, and $q_2$ is odd, and $r_1>r_2>0$. We have
\begin{align*}
t_{n,m}
&= \prod_{j=1}^{r_0} \dd (\dd\rr)^{1+\floor{jr_1}{r_0}-\floor{(j-1)r_1}{r_0}} \\
&= \prod_{j=1}^{r_0} \dd^2 \rr (\dd\rr)^{\floor{jr_1}{r_0}-\floor{(j-1)r_1}{r_0}} \\
&= \prod_{j=1}^{r_1} (\dd^2 \rr)^{\ceil{jr_0}{r_1}-\ceil{(j-1)r_0}{r_1}} \dd \rr\\
&= \prod_{j=1}^{r_1} (\dd^2 \rr)^{q_2 + \ceil{jr_2}{r_1}-\ceil{(j-1)r_2}{r_1}} \dd \rr\\
&= \prod_{j=1}^{r_1} (\dd^2 \rr)^{1 + \ceil{jr_2}{r_1}-\ceil{(j-1)r_2}{r_1}} \dd \rr\\
&= \prod_{j=1}^{r_1} (\dd^2 \rr)^{\ceil{jr_2}{r_1}-\ceil{(j-1)r_2}{r_1}} \dd^2 \rr \dd \rr\\
&= \prod_{j=1}^{r_2} \dd^xc2 \rr (\dd^2 \rr \dd \rr)^{\floor{jr_1}{r_2}-\floor{(j-1)r_1}{r_2}} \\
&\sim \prod_{j=1}^{r_2} \dd \rr \dd \rr^{3(q_3 + \floor{jr_3}{r_2}-\floor{(j-1)r_3}{r_2})} \\
&= \prod_{j=1}^{r_2} \rr^3 \dd^3 \rr^3 \rr^{3(q_3 + \floor{jr_3}{r_2}-\floor{(j-1)r_3}{r_2})} \\
&\sim \prod_{j=1}^{r_2} \rr \dd \rr \rr^{q_3 + \floor{jr_3}{r_2}-\floor{(j-1)r_3}{r_2}} \\
&\sim \prod_{j=1}^{r_2} \dd \rr^{q_3 + 2 + \floor{jr_3}{r_2}-\floor{(j-1)r_3}{r_2}} \\
&\sim \prod_{j=1}^{r_2} \dd \rr^{2 + \floor{jr_1}{r_2}-\floor{(j-1)r_1}{r_2}} \\
&\sim \prod_{j=1}^{r_2} \dd^{2+\floor{jr_1}{r_2}-\floor{(j-1)r_1}{r_2}} \rr \\
&= \prod_{j=1}^{r_2} \dd^{\floor{j(r_1+2r_2)}{r_2}-\floor{(j-1)(r_1+2r_2)}{r_2}} \rr \\
&\sim \prod_{j=1}^{r_2} \dd^{\ceil{j(r_1+2r_2)}{r_2}-\ceil{(j-1)(r_1+2r_2)}{r_2}} \rr \\
&= t_{r_2,r_1+2r_2}.
\end{align*}

Suppose $q_0=1$ and $q_1=1$, and $q_2$ is odd, and $r_1>r_2=0$. Note that if $r_2=0$ we must have $r_1 = 1$. Then we have
\begin{align*}
t_{n,m}
&= \prod_{j=1}^{r_0} \dd (\dd\rr)^{1+\floor{jr_1}{r_0}-\floor{(j-1)r_1}{r_0}} \\
&= \prod_{j=1}^{r_0} \dd^2 \rr (\dd\rr)^{\floor{jr_1}{r_0}-\floor{(j-1)r_1}{r_0}} \\
&= \prod_{j=1}^{r_1} (\dd^2 \rr)^{\ceil{jr_0}{r_1}-\ceil{(j-1)r_0}{r_1}} \dd \rr\\
&= \prod_{j=1}^{r_1} (\dd^2 \rr)^{q_2 + \ceil{jr_2}{r_1}-\ceil{(j-1)r_2}{r_1}} \dd \rr\\
&= \prod_{j=1}^{r_1} (\dd^2 \rr)^{q_2} \dd \rr\\
&= \prod_{j=1}^{r_1} (\dd^2 \rr) \dd \rr\\
&= \dd^2 \rr \dd \rr \\
&= t_{2,3}.
\end{align*}
\end{proof}

These reductions allow us to determine exactly when a graph has 2 diagonals.

\begin{restatable}{theorem}{twodiag}
Let $n,m>1$ with $\gcd(n,m)=1$. Then $\diag(n,m) = 2$ if and only if $nm$ is odd.
\end{restatable}

\begin{proof}
As the reductions can be applied to any graph $G_{n,m,2}$ with $r_0>0$ and $q_0>1$ or $r_1>0$ or $q_1>3$, the base case ordered pairs are a subset of $\{(1,m) \mid m>0\} \cup \{(2,3),(3,4)\}$. But if $m>4$ then $(1,m)$ can be reduced to $(1,m-4)$. Hence the base cases are $(1,1)$, $(1,2)$, $(1,3)$, $(1,4)$, $(2,3)$, and $(3,4)$.

Every reduction $(n,m) \to (n',m')$ preserves $n+m \mod{2}$, so some pair $(n,m)$ has both sizes odd if and only if the base case to which the pair can be reduced also has both sizes odd. The only two base cases with both sizes odd are $(1,1)$ and $(1,3)$, and it can be checked that these are exactly the base cases with $2$ diagonals.
\end{proof}


\subsection{Ternary Tree Classification}\label{subsec:ternary}

Our second set of reductions are movements in a tree of coprime pairs; we prove these reductions using the first set.

\begin{defn}
For any coprime pair $(m,n)$ with $m>n$, let:
\begin{singlespacing}
\begin{itemize}
\item $\gamma(m,n) = (2m-n,m)$,
\item $\delta(m,n) = (2m+n,m)$, and
\item $\lambda(m,n) = (m+2n,n)$.
\end{itemize}
\end{singlespacing}
\end{defn}

These three functions generate a ternary tree rooted at the pair $(2,1)$: each pair $(m,n)$ in the tree has children $\gamma(m,n)$, $\delta(m,n)$, and $\lambda(m,n)$. It is a well-known fact that every coprime pair $(m,n)$ with $m>n$ and $m+n$ odd appears exactly once in this tree. We refer to such pairs as \textit{even-odd} pairs.

\begin{defn}
Let the \textit{tree string} $p_{m,n}$ of an even-odd pair $(m,n)$ be the unique string consisting of characters $\gamma$, $\delta$, and $\lambda$, such that $p_{m,n}(2,1) = (m,n)$.
\end{defn}

We can compute the tree string of any even-odd pair recursively.

\begin{restatable}{prop}{treestring}\label{prop:treestring}
Let $(m,n)$ be an even-odd pair. If $(m,n) = (2,1)$ then $p_{m,n} = \epsilon$, or the empty string. Otherwise there are three cases:

If $m < 2n$, then $p_{m,n} = \gamma p_{n,2n-m}$.

If $2n \leq m < 3n$, then $p_{m,n} = \delta p_{n,m-2n}$.

If $m \geq 3n$, then $p_{m,n} = \lambda p_{m-2n,n}$.
\end{restatable}

\begin{proof}
This follows trivially from the fact that every even-odd pair occurs exactly once in the tree, by considering the domains of the inverse operations $\gamma^{-1}$, $\delta^{-1}$, and $\lambda^{-1}$.
\end{proof}

Now we wish to classify tree strings by the number of diagonals in associated graphs. We say that $(m,n) \sim (m',n')$ if $\diag(m,n) = \diag(m',n')$.

Every tree string $p$ can be simplified, using the following reductions, to some canonical tree string $p_c$ such that $p(2,1) \sim p_c(2,1)$.

\begin{restatable}{prop}{canon}\label{prop:canon}
Let $(m,n)$ be an even-odd pair. Then we have:
\begin{singlespacing}
\begin{enumerate}
\item $\delta(m,n) \sim \gamma(m,n)$.
\item $\lambda \kappa(m,n) \sim (m,n)$ for any $\kappa \in \{\gamma,\delta,\lambda\}$.
\item $\gamma \delta (m,n) \sim \lambda (m,n)$.
\item $\gamma \lambda (m,n) \sim \gamma (m,n)$.
\item $\gamma \gamma \kappa (m,n) \sim (m,n)$ for any $\kappa \in \{\gamma,\delta,\lambda\}$.
\end{enumerate}
\end{singlespacing}
\end{restatable}

\begin{proof}
We use the previously proved reductions to prove these ones. Suppose $m = q_0n+r_0$ where $0 \leq r_0 < n$.

\begin{enumerate}
\item $\delta(m,n) = (2m+n,m) \sim (4m-(2m+n),m) = (2m-n,m) = \gamma(m,n).$
\item $\lambda \gamma (m,n) = (4m-n,m) \sim (n,m) \sim (m,n).$

$\lambda \delta (m,n) = (4m+n,m) \sim (n,m) \sim (m,n).$

$\lambda^2 (m,n) = (m+4n,n) \sim (m,n).$

\item $\gamma \delta (m,n) = (3m+2n,2m+n) \sim (m+2n,n) = \lambda(m,n).$

\item If $q_0$ is even, $\gamma \lambda (m,n) = (2m+3n,m+2n) \sim (r_0,n+2r_0)$ and $\gamma (m,n) = (2m-n,m) \sim (r_0,n+2r_0)$. If $q_0=1$, $\gamma \lambda (m,n) \sim (n+r_0,n+2r_0) = \gamma(m,n)$. And if $q_0>1$ is odd, $\gamma \lambda (m,n) \sim (n+r_0,n+2r_0) \sim \gamma(m,n)$.

\item If $q_0 = 1$, then $\gamma^2 \gamma (m,n) = (4m-3n,3m-2n) = (n+4r_0,n+r_0) \sim (n+r_0,n) = (m,n)$. Otherwise, $\gamma^2 \gamma (m,n) = (4m-3n,3m-2n) \sim (n,m) \sim (m,n)$.

Applying Reduction 1, $\gamma^2 \delta (m,n) = \gamma^2 \gamma (m,n) \sim (m,n)$.

And finally, $\gamma^2 \lambda (m,n) = (3m+4n,2m+3n) \sim (n,m) \sim (m,n)$.

\end{enumerate}
\end{proof}


It can be seen that the canonical tree strings---the strings which cannot be reduced further---are $\gamma$, $\gamma^2$, $\lambda$, and the empty string $\epsilon$. The canonical tree string of a graph can be determined by feeding the tree string into an automaton (Figure~\ref{fig:automaton}).

\begin{figure}[h]
\centering
\includegraphics[width=0.35\textwidth]{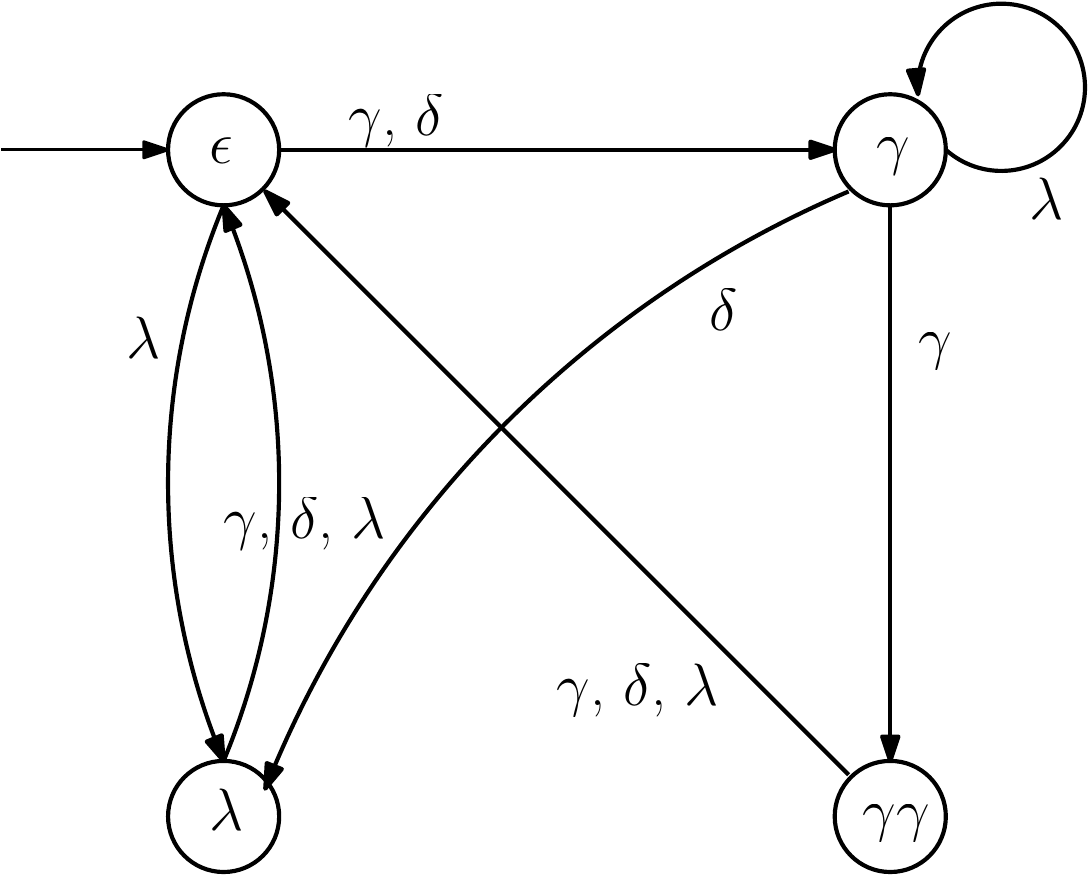}
\caption{The transitions between canonical tree strings}\label{fig:automaton}
\end{figure}

Combining several earlier results with these observations yields the following simple and efficient algorithm to find the number of diagonals in an arbitrary graph $G_{n,m,2}$:

\begin{singlespacing}
\begin{enumerate}
\item Let $g = \gcd(m,n)$.
\item Divide $g$ out of $m$ and $n$.
\item If $m$ and $n$ are both odd, return $2 \cdot g$.
\item Apply Proposition~\ref{prop:treestring} to find the tree-string $p_{m,n}$.
\item Apply Proposition~\ref{prop:canon} to reduce $p_{m,n}$ to a canonical tree string $p_c$.
\item Return $g \cdot \diag(p_c(2,1))$.
\end{enumerate}
\end{singlespacing}

Note that this algorithm has time complexity $\mathcal{O}(\log(n))$.




\section{Determining Hamiltonicity in Special Cases}\label{sec:speccases}

\begin{restatable}{prop}{onediagnotham}
Let $n,m > 1$. If $G_{n,m,2}$ has one diagonal then it is not Hamiltonian.
\end{restatable}

\begin{proof}
By Lemma~\ref{lemma:rightdown}, the Hamiltonian cycle edges out of all cells in $G_{n,m,2}$ must point in the same direction. If the edges point up, then the cycle visits exactly two columns, so $m=1$. Similarly, if the edges point right, then $n=1$. In both cases we achieve a contradiction, so $G_{n,m,2}$ is not Hamiltonian.
\end{proof}

The converse, however, is not true. See \ref{app:compute} for examples.


Having one diagonal not only correlates negatively with the existence of a Hamiltonian cycle, as in the above proposition, but also correlates positively, as in the proposition below.

\begin{restatable}{prop}{onediagham}
Let $n$ and $m$ be positive integers. If $G_{n,m,2}$ has one diagonal, then $G_{2n,2m,2}$ is Hamiltonian.
\end{restatable}

\begin{proof}
There are two diagonals in $G_{2n,2m,2}$. Let the one be oriented up and the other be oriented right. As they are both similar to the single diagonal in $G_{n,m,2}$, they must both touch the right and left upper boundaries $m$ times each, and the upper and lower right boundaries $n$ times each. Hence the parameters of the associated link are $(m,m,n,n)$. But the permutation graph of the diagonal edges of any graph can also be associated with a link; diagonals are right and down, but this is inconsequential. For $G_{n,m,2}$, the parameters are $(m,m,n,n)$ since the diagonal touches the right lower edge $m$ times, and so forth. But since the permutation graph of the diagonals of $G_{n,m,2}$ consists of one cycle, the $(m,m,n,n)$ graph is a knot.
\end{proof}


For graphs $G_{n,n,2}$, there is a relatively simple construction of a Hamiltonian cycle.

\begin{restatable}{prop}{squareham}
Let $n>0$. Then the $(n,n)$ grid contains a Hamiltonian cycle.
\end{restatable}

\begin{proof}
We construct a Hamiltonian cycle as follows. Let the starting cell be at row $n$ and column $0$. Then apply $r$ for $4n-1$ steps, and $u$ once. Repeat $n$ times. It can be shown that every vertex is visited exactly once before returning to the starting cell.
\end{proof}

For values of $n$ where every graph $G_{n,m,2}$ with multiple diagonals has a Hamiltonian cycle, there is a method of classifying which graphs are Hamiltonian by producing constructions for the Hamiltonian graphs and showing that the remaining graphs each have one diagonal. We describe the case when $n=2$.

\begin{restatable}{prop}{ntwoham}
Let $m>0$ with $m \equiv 0,1,2,4,6,7 \pmod{8}$. Then the $(2,m)$ grid contains a Hamiltonian cycle.
\end{restatable}

\begin{proof}
We use the coordinate system in this proof where the right quadrants are placed below the left quadrants, producing a $8 \times m$ grid.

If $m \equiv 0 \pmod{8}$, we set the direction of a cell $(r,c)$ to be $r$ iff $r-c \equiv 1 \pmod{2}$.

If $m \equiv 1 \pmod{8}$, we set the direction of a cell $(r,c)$ to be $r$ iff $r-c \equiv 1,2,6,7 \pmod{8}$.

If $m \equiv 2 \pmod{8}$, we set the direction of a cell to be $r$ iff $r-c \equiv 3 \pmod{8}$.

If $m \equiv 4 \pmod{8}$, we set the direction of a cell to be $r$ iff $r-c \equiv 0 \pmod{8}$.

If $m \equiv 6 \pmod{8}$, we set the direction of a cell to be $r$ iff $r-c \equiv 1 \pmod{8}$.

If $m \equiv 7 \pmod{8}$, we set the direction of a cell to be $r$ iff $r-c \equiv 0,1,2,5 \pmod{8}$.
\end{proof}


For the second piece of the argument, it would suffice, by Proposition~\ref{prop:modfour}, to check base cases only, but we describe a different, self-contained argument, as the method used is interesting.

\begin{restatable}{prop}{ntwonotham}
Let $m>0$ with $m \equiv 3,5 \pmod{8}$. Then the $(2,m)$ grid contains only one diagonal, and thus has no Hamiltonian cycle.
\end{restatable}

\begin{proof}
We use the standard coordinate system in this proof.

For each $d$ with $-3 \leq d \leq 2m-1$, define the \textit{diagonal segment} $S_d$ by $$S_d = \{ (r,c) \mid c - r = d \}.$$

Note that each diagonal is composed of several diagonal segments.

Let $f$ be the bijection on diagonal segment indices such that $S_{f(d)}$ is the segment immediately following $S_d$ within the same diagonal. Then

$$f(d) = \begin{cases}
(d + m + 4) \pmod{2m} & -3 \leq d \leq 2m-5 \\
-2 & d = 2m-4 \\
-1 & d = 2m-3 \\
M & d = 2m-2 \\
-3 & d = 2m-1
\end{cases}$$

We wish to show that for any index $d$, there is some $i \geq 0$ such that $f^i(d) = 0$. If every application of $f$ added $m+4$, this would be simple.

Thus, we must consider the abnormal values of $d$: viz. $2m-1$, $2m-2$, $2m-3$, or $2m-4$. We note that $f^2(2m-1) = m+1$, and $f(2m-2) = m$, and $f^2(2m-3) = m+3$, and $f^2(2m-4) = m+2$.

\textbf{Case 1:} $m = 8k+3$.

Then:
\begin{itemize}
\item $2m-1 = 16k+5$ leads to $m+1 = 8k+4$.
\item $2m-2 = 16k+4$ leads to $m = 8k+3$.
\item $2m-3 = 16k+3$ leads to $m+3 = 8k+6$.
\item $2m-4 = 16k+2$ leads to $m+2 = 8k+5$.
\end{itemize}

\begin{table}[h]
\centering
\begin{tabular}{|C|CCCCC|} \hline
     & (2m-4)&(2m-3) &(2m-2) &(2m-1) & (0)  \\
     & 16k+2 & 16k+3 & 16k+4 & 16k+5 & 0    \\ \hline
8k+3 & 16k+5 & 2k    & 4k+1  & 6k+2  & 8k+3 \\
8k+4 & 14k+4 & 16k+5 & 2k    & 4k+1 & 6k+2  \\
8k+5 & 12k+3 & 14k+4 & 16k+5 & 2k   & 4k+1  \\
8k+6 & 10k+2 & 12k+3 & 14k+4 & 16k+5& 2k    \\ \hline
\end{tabular}
\caption{Distance to abnormal values}\label{tab:3mod8}
\end{table}

For each of these values which follow from abnormal values, we can determine how many additions of $m+4$ would yield either $0$ or another abnormal value (see Table \ref{tab:3mod8}). The value with the smallest distance will be reached first. Hence, it can be seen that each abnormal value will reach $0$ after several applications of $f$, possibly passing through other abnormal values along the way.

Now we may consider an arbitrary $d$. Suppose it never reaches an abnormal value. Then every application of $f$ adds $m+4$. But 

\begin{align*}
\gcd(m+4,2m)
&= \gcd(8k+7,16k+6) \\
&= \gcd(8k+7,8k-1) \\
&= \gcd(8,8k-1) \\
&= 1.
\end{align*}

Thus every value between $0$ and $2m-1$ is reached. This is a contradiction, so an abnormal value is eventually reached, and therefore $0$ is eventually reached. Thus every diagonal segment is in the same diagonal as $S_0$.

\textbf{Case 2:} $m = 8k+5$.

This case can be proved in an almost identical format.
\end{proof}


\section{Future Work}\label{sec:future}

The main direction for future research we see is trying to completely classify which graphs $G_{n,m,2}$ are Hamiltonian, or at least find a polynomial-time algorithm, as an improvement upon our pseudo-polynomial time algorithm. There are several other generalizations which we think may be worth trying: trying to extend our results to tori with more than $2$ holes, or relaxing the restriction that the boundaries are cut into equal segments. More specific problems with which we have been grappling are proving a smaller periodicity even when $n$ and $m$ are not coprime, and trying to classify the relatively sparse subset of ordered pairs $(n,m)$ where $G_{n,m,2}$ has multiple diagonals but is not Hamiltonian.

\section{Acknowledgments} 

I would like to thank Chiheon Kim for his invaluable advice. I would also like to thank Tanya Khovanova for 
her suggestions and review, as well as Ankur Moitra, David Jerison and Slava Gerovitch for their guidance on 
this project. Additionally, I would like to thank Jenny Sendova for her comments on this paper, and 
suggestions for visualizing a two-holed torus. Finally I would like to thank the Center for Excellence in Education, the Research Science Institute, and MIT for their support.

\pagebreak

\appendix
\gdef\thesection{Appendix \Alph{section}}

\section{Computational Data}\label{app:compute}  

\begin{table}[h]
\centering
\begin{tabular}{|C|C|C|C|} \hline
$(5,19)$  & $(5,41)$  &           &           \\\hline
$(7,27)$  & $(7,29)$  & $(7,55)$  & $(7,57)$  \\\hline
$(11,53)$ &           &           &           \\\hline
$(13,29)$ & $(13,31)$ & $(13,43)$ & $(13,47)$ \\\hline
$(17,31)$ & $(17,37)$ & $(17,39)$ & $(17,55)$ \\\hline
$(19,47)$ & $(19,53)$ &           &           \\\hline
$(20,29)$ &           &           &           \\\hline
$(25,43)$ &           &           &           \\\hline
$(27,43)$ & $(27,49)$ & $(27,59)$ &           \\\hline
$(31,37)$ &           &           &           \\\hline
$(32,59)$ &           &           &           \\\hline
$(33,43)$ & $(33,53)$ &           &           \\\hline
$(35,59)$ &           &           &           \\\hline
$(36,53)$ & $(36,59)$ &           &           \\\hline
$(41,56)$ &           &           &           \\\hline
$(53,56)$ &           &           &           \\\hline
\end{tabular}
\caption{Coprime pairs $(n,m)$ with $n<m \leq 60$ such that $G_{n,m,2}$ has multiple diagonals but is not Hamiltonian}\label{table:weirdvalues}
\end{table}

As a side note, based on significant computational evidence (we tested values of $h$ from $100$ to $8000$ and the error approaches zero as $h$ increases), we hazard a conjecture about the frequency of graphs with 1, 2, or 3 diagonals.
	
\begin{conj}
Let $P_{k,h}$ denote the fraction of coprime pairs $(m,n)$ with $h \geq m>n$, such that $\diag(m,n) = k$. Then $$\lim_{h \to \infty} P_{1,h} = \frac{4}{9},$$
$$\lim_{h \to \infty} P_{2,h} = \frac{1}{3},$$
$$\lim_{h \to \infty} P_{3,h} = \frac{2}{9}.$$
\end{conj}

We suspect that the denominators are powers of $3$ due to the ternary structure of the tree of coprime pairs.

\end{document}